\documentclass[11pt,a4paper]{article}
\usepackage[utf8]{inputenc}
\usepackage{amsmath}
\usepackage{amsfonts}
\usepackage{amssymb}
\usepackage{amsthm}

\usepackage{tikz}
\usepackage[font={small,it}]{caption}

\theoremstyle{plain}
\newtheorem{theorem}{Theorem}
\newtheorem{lemma}[theorem]{Lemma}
\newtheorem{corollary}[theorem]{Corollary}
\newtheorem{proposition}[theorem]{Proposition}

\newtheorem{claim}[theorem]{Claim}

\theoremstyle{definition}
\newtheorem{definition}[theorem]{Definition}
\newtheorem{construction}[theorem]{Construction}

\usepackage{graphicx}

\numberwithin{equation}{section}
\numberwithin{theorem}{section}

\DeclareMathOperator{\Sym}{Sym}
\DeclareMathOperator{\cp}{cp}

\newcommand\cE{\mathcal{E}}
\newcommand\cF{\mathcal{F}}
\newcommand\cB{\mathcal{B}}

\begin{document}

\title{Counting racks of order $n$}
\author{Matthew Ashford and Oliver Riordan%
\thanks{Mathematical Institute, University of Oxford, Andrew Wiles Building, Radcliffe Observatory Quarter, Woodstock Road, Oxford OX2 6GG, UK.
E-mail: {\tt matthew.ashford1@btinternet.com, riordan@maths.ox.ac.uk}.}}
\date{May 20, 2017}
\maketitle

\begin{abstract}
A rack on $[n]$ can be thought of as a set of maps $(f_x)_{x \in [n]}$, where each $f_x$ is a permutation of $[n]$ such that $f_{(x)f_y} = f_y^{-1}f_xf_y$ for all $x$ and $y$. In 2013, Blackburn showed that the number of isomorphism classes of racks on $[n]$ is at least $2^{(1/4 - o(1))n^2}$ and at most $2^{(c + o(1))n^2}$, where $c \approx 1.557$; in this paper we improve the upper bound to $2^{(1/4 + o(1))n^2}$, matching the lower bound. The proof involves considering racks as loopless, edge-coloured directed multigraphs on $[n]$, where we have an edge of colour~$y$ between $x$ and $z$ if and only if $(x)f_y = z$, and applying various combinatorial tools.
\end{abstract}

\section{Introduction}

A \emph{rack} is a pair $(X, \triangleright)$, where $X$ is a non-empty set and $\triangleright:X \times X \to X$ is a binary operation such that:
\begin{enumerate}
\item For any $y, z \in X$, there exists $x \in X$ such that $z = x \triangleright y$;
\item Whenever we have $x,y,z \in X$ such that $x \triangleright y = z \triangleright y$, then $x=z$;
\item For any $x,y,z \in X$, $(x \triangleright y) \triangleright z = (x \triangleright z) \triangleright (y \triangleright z)$.
\end{enumerate}
If $X$ is finite, we call $|X|$ the \emph{order} of the rack. Note that conditions 1 and 2 above are equivalent to the statement that for each $y$, the map $x \mapsto x \triangleright y$ is a bijection on $X$.

As mentioned by Blackburn in \cite{mainpaper}, racks originally developed from correspondence between J.H. Conway and G.C. Wraith in 1959, while more specific structures known as \emph{quandles} (which are racks such that $x \triangleright x = x$ for all $x$) were introduced independently by Joyce \cite{quanhist1} and Matveev \cite{quanhist2} in 1982 as invariants of knots. Fenn and Rourke \cite{history} provide a history of these concepts, while Nelson \cite{widerhist} gives an overview of how these structures relate to other areas of mathematics.

As a first example, note that for any set $X$, if we define $x \triangleright y = x$ for all $x,y \in X$, then we obtain a rack, known as the \emph{trivial rack} $T_X$. If $G$ is a group and $\triangleright:G \times G \to G$ is defined by $x \triangleright y := y^{-1}xy$, then the resulting quandle $(G, \triangleright)$ is known as a \emph{conjugation quandle}. For a further example, let $A$ be an Abelian group and $\tau \in \text{Aut}(A)$ be an automorphism. If we define a binary operation $\triangleright: A \times A \to A$ by $x \triangleright y = (x)\tau + y - (y)\tau = (x-y)\tau + y$ then $(G, \triangleright) = (A, \tau)$ is an \emph{Alexander quandle} or \emph{affine quandle}.\footnote{Throughout the paper, we write maps on the right.}

Let $(X, \triangleright)$ and $(X', \triangleright')$ be racks. A map $\phi:X \to X'$ is a \emph{rack homomorphism} if $(x \triangleright y)\phi = (x)\phi \triangleright' (y)\phi$ for all $x,y \in X$. A bijective homomorphism is an \emph{isomorphism}.

We will only be concerned with racks up to isomorphism. If $(X, \triangleright)$ is a rack of order $n$, then it is clearly isomorphic to a rack on $[n]$, so we will take $[n]$ to be our underlying ground set. We will denote the set of all racks on $[n]$ by $\mathcal{R}_n$, and the set of all isomorphism classes of racks of order $n$ by $\mathcal{R}'_n$, so $|\mathcal{R}'_n| \leqslant |\mathcal{R}_n|$.

There have been several published results concerning the enumeration of quandles of small order; Ho and Nelson \cite{algo1} and Henderson, Macedo and Nelson \cite{algo2} enumerated the isomorphism classes of quandles of order at most 8, while work of Clauwens \cite{<35no1} and Vendramin \cite{<35no2} gives an enumeration of isomorphism classes of quandles of order at most 35 whose operator group is transitive (the operator group is defined in Section \ref{graphsec}). Recently, Jedli{\v{c}}ka, Pilitowska, Stanovsk{\`y} and Zamojska-Dzienio \cite{medial} gave an enumeration of medial quandles (a type of affine quandle) of order at most 13. As far as we are aware, the only previous asymptotic enumeration result for general racks was due to Blackburn \cite{mainpaper}, giving lower and upper bounds for $|\mathcal{R}'_n|$ of $2^{(1/4 + o(1))n^2}$ and $2^{(c + o(1))n^2}$ respectively, where $c \approx 1.557$. Theorem 8.2 of \cite{medial} improves the upper bound to $2^{(1/2 + o(1))n^2}$ in the case of medial quandles; the authors then conjecture an upper bound of $2^{(1/4 + o(1))n^2}$ under the same restriction. The main result of this paper proves this upper bound for general racks, and hence for (medial) quandles.
\begin{theorem}
\label{mainthm}
Let $\epsilon > 0$. Then for all sufficiently large integers $n$,
\[ 2^{(1/4 - \epsilon)n^2} \leqslant |\mathcal{R}'_n| \leqslant |\mathcal{R}_n| \leqslant 2^{(1/4 + \epsilon)n^2}. \]
\end{theorem}
The lower bound follows from the construction in Theorem 4 of \cite{mainpaper}; our focus is on the upper bound. The bound given in Theorem 12 of \cite{mainpaper} was obtained by applying group theoretic results to the operator group associated with a rack; in our arguments we apply combinatorial results to a graph associated with a rack. This graph is defined in the next section.

\section{Graphical representations of racks}
\label{graphsec}

For any rack $(X, \triangleright)$, we can define a set of bijections $(f_y)_{y \in X}$ by setting $(x)f_y = x \triangleright y$ for all $x$ and $y$. The following well-known result (see for example, \cite{history}, \cite{mainpaper}) gives the correct conditions for a collection of maps $(f_y)_{y \in X}$ to define a rack.
\begin{proposition}
\label{mapdef}
Let $X$ be a set and let $(f_x)_{x \in X}$ be a collection of functions each with domain and co-domain $X$. Define a binary operation $\triangleright:X \times X \to X$ by $x \triangleright y := (x)f_y$. Then $(X, \triangleright)$ is a rack if and only if $f_y$ is a bijection for each $y \in X$, and for all $y,z \in X$ we have
\begin{equation}
\label{maps}
f_{(y)f_z} = f_z^{-1}f_yf_z. 
\end{equation}
\end{proposition}
\begin{proof}
As noted earlier, conditions 1 and 2 in the definition of a rack hold if and only if each $f_y$ is a bijection, so it remains to show that condition 3 is equivalent to \eqref{maps}. This is essentially a reworking of the definition; we omit the simple details.
\end{proof}
In the light of Proposition~\ref{mapdef}, we can just as well define a rack on a set $X$ by the set of maps $(f_y)_{y \in X}$, providing they are all bijections and satisfy \eqref{maps}. We will move freely between the two definitions, with $x \triangleright y = (x)f_y$ for all $x,y \in X$ unless otherwise stated.

The \emph{operator group} of a rack is the subgroup of $\Sym(X)$ generated by $(f_y)_{y \in X}$. The following standard lemma (see for example Lemma 6 of \cite{mainpaper}) shows that Proposition \ref{mapdef} can be extended to elements of the operator group.
\begin{lemma}
\label{opgroup}
Let $(X, \triangleright)$ be a rack and let $G$ be its operator group. Then for any $y \in X$ and $g \in G$, $f_{(y)g} = g^{-1}f_yg$. \qed
\end{lemma}

Any rack on $X$ can be represented by a directed multigraph on $X$; we give each vertex a colour and then put a directed edge of colour $i$ from vertex $j$ to vertex $k$ if and only if $(j)f_i = k$. We then remove all loops from the graph; i.e.\ if $(j)f_i = j$ we don't have an edge of colour $i$ incident to $j$.

It will be helpful to recast the representation of racks by directed multigraphs in a slightly different setting. Let $V$ be a finite set and let $\sigma \in \Sym(V)$; then we can define a simple, loopless directed graph $G_\sigma$ on $V$ by setting $\overrightarrow{uv} \in \overrightarrow{E}(G_\sigma)$ if and only if $u \neq v$ and $(u)\sigma = v$. By considering the decomposition of $\sigma$ into disjoint cycles, we see that $G_\sigma$ consists of a collection of disjoint directed cycles, isolated double edges corresponding to cycles of length two, and some isolated vertices. We can now extend this definition to the case of multiple permutations in a natural way.
\begin{definition}
\label{multinvodef}
Suppose $\Sigma = \lbrace \sigma_1, \ldots, \sigma_k \rbrace \subseteq \Sym(V)$ is a set of permutations on a set $V$. Define a directed, loopless multigraph $G_\Sigma = (V,\overrightarrow{E})$ with a $k$-edge-colouring by putting a directed edge of colour $i$ from $u$ to $v$ if and only if $u \neq v$ and $(u)\sigma_i = v$.

We also define the \emph{reduced} graph $G_\Sigma^0$ to be the directed graph on $V$ obtained by letting $e = \overrightarrow{uv} \in \overrightarrow{E}(G_\Sigma^0)$ if and only if there is at least one directed edge from $u$ to $v$ in $G_\Sigma$. 
\end{definition}
Note that the reduced graph contains at most two edges between any $u,v \in V$, namely $\overrightarrow{uv}$ and $\overrightarrow{vu}$. Also observe that if $\Sigma' \subseteq \Sigma$, then $G_{\Sigma'}$ is a subgraph of $G_\Sigma$.

Before continuing, let us clarify some terminology. A \emph{path} in a directed multigraph $G$ need not respect the orientation of the edges, so for $x,y \in V(G)$, there is an $xy$-path in $G$ if and only if there is an $xy$-path in the underlying undirected graph. A \emph{component} of a directed graph $G$ is defined to be a component of the underlying undirected graph. For $x,y \in V(G)$, a \emph{directed} $xy$-path is a sequence of vertices $x = x_0, \ldots, x_r = y$ such that $\overrightarrow{x_ix_{i+1}} \in \overrightarrow{E}(G)$ for all $i$.

Now let us return specifically to racks.
\begin{definition}
\label{subgdef}
Let $R = (X, \triangleright)$ be a rack, and let $(f_y)_{y \in X}$ be the associated maps. For any $S \subseteq X$, define $\Sigma_S = \lbrace f_y \mid y \in S \rbrace$. Then by $G_S$ we mean the directed multigraph $G_{\Sigma_S}$ in the sense of Definition \ref{multinvodef}; $G_S$ thus has an associated $|S|$-edge-colouring, although if $|S|=1$ we may not necessarily consider $G_S$ as being coloured. We will also write $G_R = G_X$, indicating the graph for the whole rack.
\end{definition}
When describing racks on $[n]$ in a graphical context, we may refer to elements of $[n]$ as vertices. The following two observations are straightforward but crucial.
\begin{lemma}
\label{dirpath}
Let $\Sigma \subseteq \Sym(V)$ be a family of permutations, and let $u, v \in V$ be distinct. Then there is a $uv$-path in $G_\Sigma$ if and only if there is a directed $uv$-path in $G_\Sigma$.
\end{lemma}
\begin{proof}
We need only prove the `only if' statement. Let $u = u_0, \ldots, u_t = v$ be a path in $G_\Sigma$; any edge $\overrightarrow{u_iu_{i-1}}$ is part of a directed cycle and thus can be replaced with a directed $u_{i-1}u_i$-path. Replacing each such edge gives a directed $uv$-walk; the shortest such walk is a path.
\end{proof}

\begin{lemma}
\label{comporb}
Let $\Sigma \subseteq \Sym(V)$ be a family of permutations and $U \subseteq V$. Then $U$ is an orbit of the natural action of $\langle \Sigma \rangle$ on $V$ if and only if $U$ spans a component in $G_\Sigma$.
\end{lemma}
\begin{proof}
Let $u,v \in V$. Then $u$ and $v$ are in the same orbit of the natural action if and only if there exists a sequence $(\sigma_{i_1}, \ldots, \sigma_{i_m})$ of elements from $\Sigma$ and a sequence $(\epsilon_1, \ldots, \epsilon_m) \in \lbrace -1,1\rbrace^m$ such that $v = (u)\sigma_{i_1}^{\epsilon_1} \cdots \sigma_{i_m}^{\epsilon_m}$. But this is exactly equivalent to there being a $uv$-path in $G_\Sigma$ with edges successively coloured $i_1, \ldots, i_m$, and the value of $\epsilon_i$ indicating the direction of the edge. Thus the partition of $V$ into orbits of $\langle \Sigma \rangle$ coincides with the partition of $V$ into components of $G_\Sigma$.
\end{proof}
Applying this last result to a rack $R$ on $[n]$ shows that the orbits of the operator group (in its natural action on $[n]$) coincide with the components of $G_R$.

We can illustrate these notions with a simple example. Let $(X, \triangleright)$ be a rack; then a \emph{subrack} of $(X, \triangleright)$ is a rack $(Y, \triangleright|_{Y \times Y})$\footnote{The notation $\triangleright|_{Y \times Y}$ in the above context will always be abbreviated to $\triangleright$, with the restriction to the subset $Y$ left implicit.} where $Y \subseteq X$. Thus a subset $Y \subseteq X$ forms a subrack if and only if for all $y,z \in Y$, $(z)f_y \in Y$; as each $f_y$ is a bijection we then also have that $(z)f_y^{-1} \in Y$ for all $z$ and thus $Y$ and $X \setminus Y$ are separated in the graph $G_Y$. 

\section{Outline of the proof}

In this short section we give a brief outline of how we will count the number of racks on $[n]$. We shall reveal information about an unknown rack on $[n]$ in several steps, counting the number of possibilities for the revealed information at each step. At the end the rack will have been determined completely, so we obtain an upper bound on the number of racks.

The principle behind the argument is as follows: we choose a set $T \subseteq [n]$ and reveal the maps $(f_j)_{j \in T}$. We then consider the components of the graph $G_T$; a key lemma shows that if $V$ is a set of vertices such that each component contains precisely one element from $V$, then revealing the maps $(f_v)_{v \in V}$ determines the entire rack $R$. The difficulty is in finding a set $T$ which is not too big, but such that $G_T$ has relatively few components.

We will actually need to consider two different sets $T$ and $W$. We choose a threshold $\Delta$ and first consider the set of vertices $S_{> \Delta}$ that have degree strictly greater than $\Delta$ in the underlying graph $G_R^0$; we will choose probabilistically a relatively small set $W \subseteq [n]$ such that any vertex with high degree in $G_R^0$ also has high degree in $G_W^0$. Because the degree of any vertex in $S_{> \Delta}$ is so high, the number of components of $G_W^0$ contained in $S_{> \Delta}$ is small; this allows us to determine the maps $(f_s)_{s \in S_{> \Delta}}$.

For the vertices in $S_{\leqslant \Delta}$ (those with degree at most $\Delta$ in $G_R^0$), we will construct greedily a set $T$ of a given size by adding vertices one at a time and revealing their corresponding maps, each time choosing the vertex whose map joins up the most components. It will follow that for every $j \in S_{\leqslant \Delta} \setminus T$, $f_j$ can only join up a limited number of components of $G_T$; we will reveal the restriction of $f_j$ to these components.

Because of the complex nature of this argument, we will `store' these revealed maps in a 7-tuple $I = \mathcal{I}(R)$, and then count the racks consistent with $I$. The main term in the argument comes from considering maps in $S_{\leqslant \Delta} \setminus T$ acting within components of $G_T$; we can control the action of these maps by first revealing some extra information corresponding to the neighbours of $T$ in $G_R^0$, which can themselves be controlled as $T$ consists of low degree vertices.

In Section \ref{infosec} we formally define the information $\mathcal{I}(R)$, which requires some straightforward graph theory; in Section \ref{detsec} we show that the number of possibilities for $I$ is at most $2^{o(n^2)}$. In Section \ref{ressec} we will complete the proof of Theorem \ref{mainthm}.

\section{Important information in a rack}
\label{infosec}

\subsection{Degrees in graphical representations of racks}

Let $R = ([n], \triangleright)$ be a rack and $T \subseteq [n]$; for each $v \in [n]$, define
\[ \Gamma_T^+(v) = \lbrace (v)f_j \mid j \in T, \; (v)f_j \neq v \rbrace, \]
so $\Gamma_T^+(v)$ is the set of vertices $w$ such that $\overrightarrow{vw} \in \overrightarrow{E}(G_T)$. If $V \subseteq [n]$, we define $\Gamma_T^+(V) := \bigcup_{v \in V} \Gamma_T^+(v)$.
\begin{definition}
\label{outdegdef}
With notation as above, the \emph{out-degree} of $v$ (with respect to $T$) is $d^+_T(v) = |\Gamma_T^+(v)|$, so $d_T^+(v)$ is the out-degree of $v$ in the simple graph $G_T^0$. 
\end{definition}
We can of course define the in-degree $d_T^-(v)$ similarly. We now show that when $S$ is a subrack, then all components of $G_S^0$ are out-regular.
\begin{lemma}
\label{reg}
Let $([n], \triangleright)$ be a rack and $(S, \triangleright)$ be a subrack, and let $C$ span a component of $G_S$, and hence also of $G_S^0$. Then for any $u,v \in C$, $d_S^+(u) = d_S^+(v)$.
\end{lemma}
\begin{proof}
First suppose that $v$ is an out-neighbour of $u$, so that there is a directed edge $\overrightarrow{uv} \in \overrightarrow{E}(G_S)$ of some colour $i \in S$, i.e.\ $(u)f_i = v$. Take an arbitrary $w \in \Gamma_S^+(u)$, so $w \neq u$ and there exists a $j \in S$ such that $w = (u)f_j$; as $S$ is a subrack, $k := (j)f_i \in S$. Now observe that
\[ (v)f_k = (v)f_{(j)f_i} = (v)f_i^{-1}f_jf_i = (u)f_jf_i = (w)f_i. \] 
Suppose for a contradiction that $(v)f_k = v$; then $(w)f_i = v = (u)f_i$, and thus $w=u$, contradicting the fact that $w \in \Gamma_S^+(u)$. Hence $(w)f_i = (v)f_k \in \Gamma_T^+(v)$, and as $w \in \Gamma_S^+(u)$ was arbitrary we have that $(\Gamma_S^+(u))f_i \subseteq \Gamma_S^+(v)$. As $f_i$ is a bijection, $d_S^+(u) = |\Gamma_S^+(u)| \leqslant |\Gamma_S^+(v)| = d_S^+(v)$.

Now let $u,v \in C$ be arbitrary; from Lemma \ref{dirpath}, there is a directed path $u = u_0, \ldots, u_r = v$ in $G_S[C]$. From above, we have that $d_S^+(u) \leqslant d_S^+(v)$. By instead considering a directed $vu$-path we have that $d_S^+(v) \leqslant d_S^+(u)$, and the result follows.
\end{proof}

\subsection{Some multigraph theory}
\label{multgss}

The construction of the information $\mathcal{I}(R)$ requires some straightforward graph theory. Here, a \emph{multigraph} $G=(V,E)$ is defined by a vertex set $V = V(G)$ and an edge multiset $E = E(G)$ of unordered pairs from $V$. For multisets $A$ and $B$, $A \uplus B$ is the multiset obtained by including each element $e$ with multiplicity $m$, where $m$ is the sum of the multiplicity of $e$ in $A$ and the multiplicity of $e$ in $B$. If $F$ is a multiset of unordered pairs from $V(G)$, $G + F$ is the multigraph with vertex set $V(G)$ and edge multiset $E(G) \uplus F$. 

In this subsection we will consider only undirected multigraphs for clarity. As paths and components of a directed multigraph are defined by the underlying undirected multigraph, all of these results remain true for directed multigraphs. We will write $\cp(G)$ for the number of components of a multigraph $G$. 

Let $G$ be a multigraph\footnote{While the following results can be formulated using just simple graphs, we will use multigraphs to be consistent with the definition of the graph of a rack.}. Then for distinct $u,v \in V(G)$ we have that $\cp(G + \lbrace uv \rbrace) = \cp(G) - 1$ if and only if there is no $uv$-path in $G$. The following result is standard.
\begin{proposition}
\label{comptower}
Let $G$ be a multigraph and $E_1$ and $E_2$ be multisets of unordered pairs of elements of $V(G)$. Then $\cp(G) - \cp(G + E_2) \geqslant \cp(G + E_1) - \cp(G + E_1 + E_2)$.
\end{proposition}
\begin{proof}
The case where $|E_2| = 1$ follows from the above observation, since there is a $uv$-path in $G + E_1$ if there is a $uv$-path in $G$. The general case now follows by induction.
\end{proof}

\begin{definition}
Let $G$ be a multigraph and $E$ be a multiset of unordered pairs of elements of $V(G)$. Let $C \subseteq V(G)$ span a component\footnote{In other words, $C$ is the vertex set of a component of $G$; the component itself is a multigraph, not just a set of vertices.} of $G$; we say that $C$ is \emph{merged by $E$} if there exist $u \in C$, $v \in V(G) \setminus C$ such that $uv$ is an edge from $E$. We denote by $M(G,E)$ the set of (vertex sets of) components of $G$ merged by $E$.
\end{definition}
Note that for multisets of edges $E$ and $F$, $M(G, E \uplus F) = M(G, E) \cup M(G,F)$. As a single edge can merge at most two components, $|M(G, \lbrace e \rbrace)| \leqslant 2$ for any unordered pair~$e$.
\begin{lemma}
\label{Jlem}
Let $G$ be a multigraph and $E$ and $\lbrace e \rbrace$ be multisets of unordered pairs of elements of $V(G)$. If $\cp(G + E + \lbrace e \rbrace) = \cp(G + E)$, then $M(G,E \uplus  \lbrace e \rbrace) = M(G,E)$.
\end{lemma}
\begin{proof}
Write $e = uv$; if $u$ and $v$ are in the same component of $G$ then $e$ is not a merging edge and $M(G, \lbrace e \rbrace) = \emptyset$, so suppose that $u$ and $v$ are in different components of $G$. Write $C$ for the vertex set of the component containing $u$ and $D$ for that containing $v$, so that $M(G, \lbrace e \rbrace) = \lbrace C,D \rbrace$.

As $\cp(G + E + \{e\}) = \cp(G + E)$ we have that $C$ and $D$ are both contained a single component of $G + E$; it follows easily that $C,D \in M(G,E)$. It follows that in all cases we have $M(G, \lbrace e \rbrace) \subseteq M(G,E)$ and thus that $M(G, E \uplus \lbrace e \rbrace) = M(G,E)$.
\end{proof}
\begin{corollary}
\label{Jbound}
Let $G$ be a multigraph and $E$ be a multiset of unordered pairs of elements of $V(G)$. Then $|M(G,E)| \leqslant 2(\cp(G) - \cp(G + E))$.
\end{corollary}
\begin{proof}
Order $E$ as $\lbrace e_1, \ldots, e_l \rbrace$ and write $a = \cp(G) - \cp(G + E)$; then there are precisely $a$ edges $e_{i_1}, \ldots, e_{i_a}$ such that $\cp(G + \lbrace e_1, \ldots, e_{i_j} \rbrace) = \cp(G + \lbrace e_1, \ldots, e_{i_j - 1} \rbrace) - 1$ for each $j$. Write $E_k = \lbrace e_1, \ldots, e_k \rbrace$ for each $k$, so we always have $M(G, E_k) = M(G, E_{k-1} \uplus \lbrace e_k \rbrace) = M(G, E_{k-1}) \cup M(G, \lbrace e_k \rbrace)$. Now consider adding the edges of $E$ in the order given to $G$; if $k \neq i_j$ for any $j$ then $\cp(G, E_k) = \cp(G, E_{k-1})$ and so from Lemma \ref{Jlem} we have that $M(G, E_k) = M(G, E_{k-1})$, while if $k = i_j$ for some $j$ we have that $|M(G, E_k)| \leqslant |M(G, E_{k-1})| + 2$. As there are only $a$ such edges it follows that $|M(G,E)| \leqslant 2a = 2(\cp(G) - \cp(G + E))$.
\end{proof}

\subsection{The information $\mathcal{I}(R)$}

We introduce the following terminology. For any rack $R = ([n], \triangleright)$ and any $\Delta$ with $1 \leqslant \Delta \leqslant n-1$, let 
\[ S_{\leqslant \Delta}(R) := \lbrace v \in [n] \mid d_R^+(v) \leqslant \Delta \rbrace \]
denote the set of all vertices with out-degree in $G_R^0$ at most $\Delta$. Write $S_{> \Delta}(R) = [n] \setminus S_{\leqslant \Delta}(R)$ for the set of all vertices with out-degree strictly greater than $\Delta$. We now show that this partition is actually a partition into subracks. 
\begin{lemma}
\label{hilosplit}
Let $R = ([n], \triangleright)$ be a rack and $1 \leqslant \Delta \leqslant n-1$. Then $(S_{>\Delta}(R), \triangleright)$ and $(S_{\leqslant \Delta}(R), \triangleright)$ are both subracks of $R$.
\end{lemma}
\begin{proof}
By Lemma \ref{reg} (with $S = R$), two vertices in the same component of $G_R$ have the same out-degree. Hence $S_{> \Delta}(R)$ and $S_{\leqslant \Delta}(R)$ are separated in $G_R$ and thus both $(S_{>\Delta}(R), \triangleright)$ and $(S_{\leqslant \Delta}(R), \triangleright)$ are subracks.
\end{proof}
Now fix 
\[ \Delta := (\log_2 n)^3, \]
so $\Delta \leqslant n-1$ for sufficiently large $n$. Given a rack $R$, we will construct a set $T(R) \subseteq S_{\leqslant \Delta}(R)$, with $|T(R)| \leqslant (\log_2 n)^2$, by the following procedure (the subgraph $G_A$, where $A \subseteq [n]$, is described in Definitions \ref{multinvodef} and \ref{subgdef}).
\begin{construction}
\label{Tconst} 
If $S_{\leqslant \Delta}(R) = \emptyset$ then $T(R) = \emptyset$. Otherwise, order the vertices of $S_{\leqslant \Delta}(R)$ as follows: first choose $u_1$ so that $\cp(G_{ \lbrace u_1 \rbrace}) \leqslant \cp(G_{\lbrace v \rbrace})$ for any $v \in S_{\leqslant \Delta}(R)$. Given a partial ordering $u_1, \ldots, u_k$, choose the next vertex $u_{k+1}$ such that $\cp(G_{\lbrace u_1, \ldots, u_k, u_{k+1} \rbrace}) \leqslant \cp(G_{ \lbrace u_1 \ldots, u_k, v \rbrace})$ for any $v \in S_{\leqslant \Delta}(R) \setminus \lbrace u_1, \ldots, u_k \rbrace$. Now take 
\[ L := \lfloor (\log_2 n)^2 \rfloor \]
and define 
\begin{equation*}
T(R) := 
\begin{cases} 
\lbrace u_1, \ldots, u_L \rbrace &\text{if } |S_{\leqslant \Delta}(R)| \geqslant L \\
S_{\leqslant \Delta}(R) & \text{otherwise.} 
\end{cases}
\end{equation*}
\end{construction}
We now introduce some more notation. For any $j \in [n]$, write $\overrightarrow{E}_j = \overrightarrow{E}(G_{ \lbrace j \rbrace})$ for the set of edges of $G_R$ of colour $j$ and
\[
 M_j := M(G_{T(R)}, \overrightarrow{E}_j)
\]
for the set of (vertex sets of) components of $G_{T(R)}$ merged by the edges of colour $j$. Note that if $j \in T(R)$ then $\overrightarrow{E}_j \subseteq \overrightarrow{E}(G_{T(R)})$ and so $M_j = \emptyset$. 

The key property of the set $T(R)$ is given in the next lemma.
\begin{lemma}
\label{mergebound}
Let $R = ([n], \triangleright)$ be a rack. Then for any $j \in S_{\leqslant \Delta}(R) \setminus T(R)$, 
\[ |M_j| \leqslant \frac{2n}{(\log_2 n)^2}. \]
\end{lemma}
\begin{proof}
Note that if $|S_{\leqslant \Delta}(R)| \leqslant L = \lfloor (\log_2 n)^2 \rfloor$, then $T(R) = S_{\leqslant \Delta}(R)$ and the statement is trivial. We will thus assume that $s = |S_{\leqslant \Delta}(R)| \geqslant L + 1$ and so $|T(R)| = L$.

For $1 \leqslant i \leqslant s$, write $H_i = G_{\lbrace u_1, \ldots, u_i \rbrace}$ and $x_i = \cp(H_{i-1}) - \cp(H_i)$, where $H_0 = G_\emptyset$ is the empty graph on $[n]$. Note that $\cp(H_i) = \cp(H_{i-1} + \overrightarrow{E}_{u_i}) \leqslant \cp(H_{i-1})$, and so $x_i \geqslant 0$ for each $i$; we also have that
\[ \sum_{i=1}^s x_i = \sum_{i=i}^s \big( \cp(H_{i-1}) - \cp(H_i) \big) = \text{cp}(H_0) - \text{cp}(H_s) \leqslant \cp(H_0) = n. \]
Now fix an $i$ and suppose that $x_i < x_{i+1}$; then
\begin{align*}
\cp(H_{i-1}) - \cp(H_i) &< \cp(H_i) - \cp(H_{i+1}) \\
&= \cp(H_{i-1} + \overrightarrow{E}_{u_i}) - \cp(H_{i-1} + \overrightarrow{E}_{u_i} + \overrightarrow{E}_{u_{i+1}}) \\
&\leqslant \cp(H_{i-1}) - \cp(H_{i-1} + \overrightarrow{E}_{u_{i+1}}),
\end{align*}
from Proposition \ref{comptower}. But then $\cp(G_{\lbrace u_1, \ldots u_{i-1}, u_i \rbrace}) > \cp(G_{ \lbrace u_1, \ldots u_{i-1}, u_{i+1} \rbrace})$, contradicting our ordering of the vertices. Hence $x_i \geqslant x_{i+1}$, and as $i$ was arbitrary we conclude that $(x_i)_{i=1}^s$ is a decreasing sequence. From this and the fact that $\sum_{i=1}^s x_i \leqslant n$ it follows that $x_{L+1} \leqslant n/(L+1) \leqslant n(\log_2 n)^{-2}$.

Now take $j \in S_{\leqslant \Delta}(R) \setminus T(R)$, so in our ordering of $S_{\leqslant \Delta}(R)$ we have $j = u_k$ for some $k > L$. By our construction, $\cp(G_{\lbrace u_1, \ldots, u_L, j \rbrace}) \geqslant \cp(G_{\lbrace u_1, \ldots, u_L, u_{L+1} \rbrace})$; noting that $G_{T(R)} = G_{\lbrace u_1, \ldots, u_L \rbrace} = H_L$, we may rewrite this as $\cp(G_{T(R)} + \overrightarrow{E}_j) \geqslant \cp(H_{L+1})$, and thus 
\[ \cp(G_{T(R)}) - \cp(G_{T(R)} + \overrightarrow{E}_j) \leqslant \cp(H_L) - \cp(H_{L+1}) = x_{L+1} \leqslant \frac{n}{(\log_2 n)^2}. \]
We can combine this with Corollary \ref{Jbound} to see that 
\[ |M_j| = |M(G_{T(R)}, \overrightarrow{E}_j)| \leqslant 2\bigl(\cp(G_{T(R)}) - \cp(G_{T(R)} + \overrightarrow{E}_j)\bigr) \leqslant \frac{2n}{(\log_2 n)^2}, \]
showing the result. 
\end{proof}
Before formally defining the information $\mathcal{I}(R)$ associated with a rack $R$, we will need some more notation. Firstly, write
\[
 T^+(R) := T(R) \cup \Gamma_R^+(T(R)).
\]
For any $j \in S_{\leqslant \Delta}(R) \setminus T(R)$, define
\[
 Y_j := \bigcup_{C \in M_j} C
\]
to be the set of \emph{vertices} in components of $G_{T(R)}$ merged by $\overrightarrow{E}_j$, and write
\[
 Z_j := [n] \setminus Y_j
\]
(see Figure \ref{Yjfig}).
\begin{figure}[tb]
\begin{center}
\begin{tikzpicture}
\draw (2.5,0.1) ellipse (1 and 0.5);
\draw (0,0) ellipse (0.5 and 0.6);
\draw (-0.1,1.8) ellipse (1.1 and 0.4);
\draw (3,2.4) ellipse (0.8 and 0.6);
\draw [fill = cyan] (4.5,0.75) ellipse (0.5 and 0.5);
\draw [fill = cyan] (6,1) ellipse (0.4 and 0.4);
\draw [fill = cyan] (6.3,2.2) ellipse (1.3 and 0.4);
\draw [fill = cyan] (4.8,3) ellipse (0.4 and 0.4);
\draw [fill = cyan] (5.4,-0.2) ellipse (0.6 and 0.4);
\draw (1.5,1.1) ellipse (0.5 and 0.5);
\draw (8.7,1) ellipse (0.5 and 1.5);

\draw [blue, thick] (4.5,0.8) -- (5.1,-0.4);
\draw [blue, thick] (5.5,0) -- (6.1,1.2);
\draw [blue, thick] (4.8,2.9) -- (5.4,2.1);
\draw [blue] (8.6,1.3) -- (8.8,-0.2);
\draw [blue] (2.5,2.2) -- (3.6,2.6);
\draw [blue] (8.8,2) -- (9,0.5);
\draw [blue] (-0.2,-0.2) -- (0.3,0.3);

\draw [blue, fill=blue] (2.3,0) circle [radius=0.1];

\node [right] at (2.35,0) {$j$};
\node [right] at (4.4,1.7) {$Y_j$};
\end{tikzpicture}
\end{center}
\caption{A representation of the components of $G_{T(R)}$, with the edges of colour $j$ in blue. Here, precisely five components (shaded light blue) are merged by the edges of colour $j$,  so $|M_j| = 5$; $Y_j$ is the set of vertices in these shaded components. If $j \in S_{\leqslant \Delta}(R) \setminus T(R)$ then only the restriction $f_j|_{Y_j}$ is included in the information $\mathcal{I}(R)$.}
\label{Yjfig}
\end{figure}
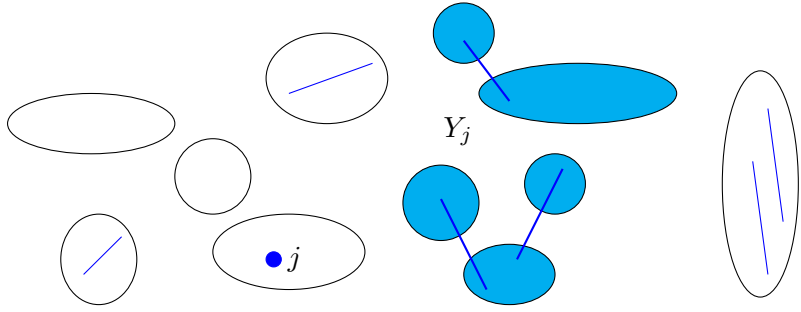
The following lemma gives the key property of the set $Z_j$, in a slightly more general setting. If $U,V \subseteq [n]$ are such that $(V)f_j = V$ for all $j \in U$, we say that $V$ is \emph{$U$-invariant}; we will write $j$-invariant instead of $\lbrace j \rbrace$-invariant.
\begin{lemma}
\label{YZsplit}
Let $R = ([n], \triangleright)$ be a rack, $W \subseteq [n]$ and $j \in [n]$. Let $C \subseteq [n]$ span a component of $G_W$, with $C \notin M(G_W, \overrightarrow{E}_j)$. Then $C$ is $j$-invariant.
\end{lemma}
\begin{proof}
Take a set $C$ as described and let $x \in C$ be arbitrary. If $(x)f_j = y \neq x$ then $\overrightarrow{xy}$ is an edge of colour $j$; as $C$ is \emph{not} merged by $\overrightarrow{E}_j$, we must have $(x)f_j = y \in C$ and (as $x$ was arbitrary and $f_j$ is a bijection) it follows that $(C)f_j = C$.
\end{proof}
Now apply this lemma with $W = T(R)$; for any $C \subseteq [n]$ spanning a component of $G_{T(R)}$ with $C \notin M_j$, $(C)f_j = C$. As $Z_j$ consists of all vertices in components of $G_{T(R)}$ not merged by $\overrightarrow{E}_j$, we have that $(Z_j)f_j = Z_j$, and thus that $(Y_j)f_j = Y_j$. 

We will now formally define the information associated with a rack $R$.
\begin{definition}
\label{calinfodef}
Let $R = ([n], \triangleright)$ be a rack and let $\Delta = (\log_2 n)^3$.
Then with notation as above let $\mathbf{M}$ be the $(|S_{\leqslant \Delta}(R)|-|T(R)|)$-tuple
\[
 \mathbf{M} := (M_j)_{j \in S_{\leqslant \Delta}(R) \setminus T(R)},
\]
where we order the vertices in some arbitrary way, and let
\[
 Y := \bigcup_{j \in S_{\leqslant \Delta}(R) \setminus T(R)}( Y_j \times \lbrace j \rbrace) \subseteq [n]^2.
\]
We define
\[ \mathcal{I}(R) := \left(S_{\leqslant \Delta}(R), \; \triangleright|_{[n] \times S_{> \Delta}}, \; T(R), \; \triangleright|_{T(R) \times [n]}, \; \triangleright|_{[n] \times T^+(R)}, \; \mathbf{M}, \; \triangleright|_Y \right). \]
\end{definition}
As $i \triangleright j = (i)f_j$, the second entry of this 7-tuple is equivalent to the set of maps $(f_j)_{j \in S_{> \Delta}(R)}$ or alternatively the graph $G_{S_{> \Delta}(R)}$. The fourth entry is equivalent to the set of maps $(f_j|_{T(R)})_{j \in [n]}$; from Lemma \ref{hilosplit}, the image of each of these maps is contained within $S_{\leqslant \Delta}(R)$. Knowing the fourth entry determines $\Gamma_R^+(T(R))$, which is necessary for the fifth entry. Note also that the fifth entry is equivalent to the set of maps $(f_j)_{j \in T^+(R)}$ and thus the graph $G_{T^+(R)}$, while the seventh entry is equivalent to $(f_j|_{Y_j})_{j \in S_{\leqslant \Delta}(R) \setminus T(R)}$.

We will think of $\mathcal{I}$ as a map from $\mathcal{R}_n$ to a set of 7-tuples; the form of this set will be considered in more detail in Section \ref{ressec}. In the next section, we show that the image $\mathcal{I}(\mathcal{R}_n)$ is small. We will do this by first considering the map $\mathcal{I}'$, where $\mathcal{I}'(R) = (\mathcal{I}_1(R), \mathcal{I}_2(R))$ and then $\mathcal{I}''$, where $\mathcal{I}''(R) = (\mathcal{I}_3(R), \ldots, \mathcal{I}_7(R))$.

\section{Determining the information $\mathcal{I}(R)$} 
\label{detsec}
\subsection{Random subsets}

The part of the argument relating to the vertices of high degree requires some probabilistic tools. In particular, we will require a result of Hoeffding \cite{genChernoff} known as the \emph{Chernoff bounds}; we will use the following, more workable version (see for example Theorems 2.1 and 2.8 and Corollary 2.3 of \cite{graphs}). 
\begin{theorem}
\label{Chercor}
Let $X_1, \ldots, X_n$ be independent random variables, each taking values in the range $[0,1]$. Let $X = \sum_{i=1}^n X_i$. Then for any $\epsilon \in [0,1]$, 
\[ \mathbb{P}(X \geqslant (1+ \epsilon)\mathbb{E}[X]) \leqslant e^{-\epsilon^2\mathbb{E}[X]/3} \]
and 
\[ \mathbb{P}(X \leqslant (1 - \epsilon)\mathbb{E}[X]) \leqslant e^{-\epsilon^2\mathbb{E}[X]/2}. \]
\qed
\end{theorem}
Let $R = ([n], \triangleright)$ be a rack. To ease notation, we write $d_v^+ = d_R^+(v)$ for any $v \in [n]$.
\begin{lemma}
\label{random}
Let $R = ([n], \triangleright)$ be a rack and let $p, \epsilon \in (0,1)$. Construct a random subset $X$ of $[n]$ by retaining each element with probability $p$, independently for all elements. Then
\begin{enumerate}
\item $\mathbb{P}(|X| \geqslant (1 + \epsilon)np) \leqslant e^{-\epsilon^2np/3}$;
\item For any $v \in [n]$ and $0 < \delta \leqslant d_v^+p$, $\mathbb{P} \big( d_X^+(v) \leqslant (1- \epsilon)\delta \big) \leqslant e^{-\epsilon^2 \delta /2}$.
\end{enumerate}
\end{lemma}
\begin{proof}
For each $j$, let $X_j = \mathbf{1}_{\lbrace j \in X \rbrace}$, so that each $X_j \sim \text{Ber}(p)$ and the variables $(X_j)_{j=1}^n$ are independent. Then $|X| = \sum_{j=1}^n X_j \sim \text{Bi}(n,p)$ and $\mathbb{E}[X] = np$, so we can apply Theorem~\ref{Chercor} to $|X|$, showing the first statement.

For the second statement, take a vertex $v \in [n]$ and let $v_1, \ldots, v_{d_v^+}$ be the out-neighbours of $v$ in $G_R^0$; for each $1 \leqslant i \leqslant d_v^+$, choose an element $j_{vi} \in [n]$ such that $(v)f_{j_{vi}} = v_i$ and put
\[ J_v := \lbrace j_{vi} \mid i = 1, \ldots, d_v^+ \rbrace. \]
The elements $j_{v1}, \ldots, j_{vd_v^+}$ are clearly distinct and so $|J_v| = d_v^+$. Now 
\[ |J_v \cap X| = \sum_{i=1}^{d_v^+} X_{j_{vi}} \sim \text{Bi}(d_v^+, p), \]
so $\mathbb{E}[|J_v \cap X|] = d_v^+p$ and we can apply Theorem \ref{Chercor} to see that for any $0 < \delta \leqslant d_v^+p$, 
\[ \mathbb{P} \big( |J_v \cap X| \leqslant (1- \epsilon)\delta \big) \leqslant \mathbb{P} \big( |J_v \cap X| \leqslant (1- \epsilon)d_v^+p \big) \leqslant e^{-\epsilon^2 d_v^+p/2} \leqslant e^{-\epsilon^2 \delta /2}. \]
Since $d_X^+(v) \geqslant |J_v \cap X|$, the second result follows.
\end{proof}

\subsection{The high degree part}

We will need the following crucial lemma.
\begin{lemma}
\label{cruclem}
Let $R = ([n], \triangleright)$ be a rack and $T \subseteq [n]$. Let $C$ span a component of $G_T$, and let $v \in C$. Let $A \subseteq [n]$ be $[n]$-invariant. Then knowledge of the maps $(f_i|_A)_{i \in T}$ and $f_v|_A$ is sufficient to determine the maps $(f_u|_A)_{u \in C}$, and further, the maps $(f_u|_A)_{u \in C}$ are conjugate in $\Sym(A)$.
\end{lemma}
\begin{proof}
Let $u \in C$, so from Lemma \ref{dirpath} there is a directed $vu$-path in $G_T$. Let the colours of the edges of this path be $i_1, \ldots, i_l$, so $(v)f_{i_1} \cdots f_{i_l} = u$ and thus from Lemma \ref{opgroup} $f_u = f_{i_l}^{-1} \cdots f_{i_1}^{-1} f_v f_{i_1} \cdots f_{i_l}$. As $A$ is $[n]$-invariant, $(A)f_{i_j} = A$ for any $j$, and thus 
\[ f_u|_A = f_{i_l}^{-1}|_A \cdots f_{i_1}^{-1}|_A f_v|_A f_{i_1}|_A \cdots f_{i_l}|_A. \]
But as each $i_j \in T$, all of these maps are determined and thus $f_u|_A$ is determined. Also note that $f_u|_A$ is conjugate to $f_v|_A$ by the map $(f_{i_1} \cdots f_{i_l})|_A$, proving the result.
\end{proof}
We will now show the first main result of this section.
\begin{proposition}
\label{bigdegprop}
Let $R = ([n], \triangleright)$ be a rack. Then for $n$ sufficiently large there exists a set $W \subseteq [n]$, with 
\begin{equation}
\label{w0eq}
|W| \leqslant w_0(n) := \frac{7}{2}\frac{n}{(\log_2 n)^{3/2}},
\end{equation}
such that $\mathcal{I}'(R) = (\mathcal{I}_1(R),\mathcal{I}_2(R))$ is determined by the sets $S_{\leqslant \Delta}(R)$ and $W$ and the maps $(f_i)_{i \in W}$.
\end{proposition}
\begin{proof}
We will construct the set $W$ by a mixture of probabilistic and deterministic arguments. Let $p = (\log_2 n)^{-3/2}$ and consider a random subset $X$ of $[n]$ as described in Lemma~\ref{random}. Let $\cE$ be the event that $|X| \leqslant 3np/2$; from item 1 of the lemma (with $\epsilon = 1/2$) we have that
\[ \mathbb{P}(\cE^\text{c}) \leqslant \mathbb{P}(|X| \geqslant 3np/2) \leqslant e^{-np/12}. \]
Since $np\to\infty$ as $n \to \infty$, it follows that $\mathbb{P}(\cE^\text{c})<1/2$ if $n$
is large enough.

Now for each $v \in [n]$ and $U \subseteq [n]$, call $v$ \emph{$U$-bad} if $d_U^+(v) \leqslant (\log_2 n)^{3/2}/2$, and let $\cB_v$ be the event that $v$ is $X$-bad, so that $\mathbb{P}(\cB_v) = \mathbb{P} \big( d_X^+(v) \leqslant (\log_2 n)^{3/2}/2 \big)$. Let $N = N(X)$ denote the number of $X$-bad vertices in $S_{> \Delta} := S_{> \Delta}(R)$, so that $N = \sum_{v \in S_{> \Delta}} \mathbf{1}_{\cB_v}$. Now $d_v^+p > \Delta p = (\log_2 n)^{3/2}$ for each $v \in S_{> \Delta}$; hence from item 2 of Lemma \ref{random} (with $\epsilon = 1/2$ and $\delta = \Delta p$), 
\[ \mathbb{E}[N] = \sum_{v \in S_{> \Delta}} \mathbb{P}(\cB_v) = \sum_{v \in S_{> \Delta}} \mathbb{P} \big( d_X^+(v) \leqslant (\log_2 n)^{3/2}/2 \big) \leqslant e^{ -(\log_2 n)^{3/2}/8} |S_{> \Delta}|. \]
Let $\cF$ be the event that $N = 0$, i.e.\ that every vertex in $S_{> \Delta}$ is $X$-good. Then from Markov's Inequality
\[ \mathbb{P}(\cF^\text{c}) = \mathbb{P}(N \geqslant 1) \leqslant \mathbb{E}[N] \leqslant |S_{> \Delta}| e^{-(\log_2 n)^{3/2}/8} \leqslant n e^{-(\log_2 n)^{3/2}/8} \to 0 \]
as $n \to \infty$. Hence $\mathbb{P}(\cF^\text{c}) <1/2$ if $n$ is large enough.

If $n$ is large enough, then $\mathbb{P}(\cE^\text{c} \cup \cF^\text{c})<1/2+1/2=1$,
and thus $\mathbb{P}(\cE \cap \cF) > 0$. Hence we can find a set $U \subseteq [n]$ such that $|U| \leqslant 3np/2$ and each vertex in $S_{> \Delta}$ is $U$-good; this means that $d_U^+(v) > (\log_2 n)^{3/2}/2$ whenever $v \in S_{> \Delta}$, or in graphical terms, that each vertex $v \in S_{> \Delta}$ is adjacent to at least $(\log_2 n)^{3/2}/2$ vertices in~$G_U^0$. 

Now from Lemma \ref{hilosplit} there are no edges from $S_{> \Delta}$ to $S_{\leqslant \Delta} = [n] \setminus S_{> \Delta}$, so $S_{> \Delta}$ is a disjoint union of vertex sets of components of $G_U^0$. Each component of $G_U^0$ contained within $S_{> \Delta}$ has size at least $(\log_2 n)^{3/2}/2 + 1$ and hence there are at most $2|S_{> \Delta}|(\log_2 n)^{-3/2}$ such components. Write the vertex sets of these components as $\lbrace C_1, \ldots, C_r \rbrace$ and take a set of vertices $V = \lbrace v_1, \ldots, v_r \rbrace$ such that $v_k \in C_k$ for each $k$; we have shown that 
\[ |V| = r \leqslant \frac{2n}{(\log_2 n)^{3/2}}.  \]
Now from Lemma \ref{cruclem} (with $T = U$ and $A = [n]$), knowledge of the maps $(f_i)_{i \in U}$ and $f_{v_k}$ is sufficient to determine the maps $(f_u)_{u \in C_k}$; applying this to each component in turn shows that knowledge of the maps $(f_i)_{i \in U}$ and $(f_v)_{v \in V}$ is sufficient to determine $(f_u)_{u \in S_{> \Delta}}$, i.e.\ the second entry of $\mathcal{I}'(R)$. So put $W = U \cup V$; as $|U| \leqslant 3n(\log_2 n)^{-3/2}/2$ we have the result.
\end{proof}
\begin{corollary}
\label{frakI1}
Let $\epsilon > 0$. There exists some positive integer $n_1$ such that for all $n \geqslant n_1$, $|\mathcal{I}'(\mathcal{R}_n)| \leqslant 2^{\epsilon n^2}$.
\end{corollary}
\begin{proof}
Take $n$ sufficiently large for the previous result to hold; then $\mathcal{I}'(R)$ is determined by $S_{\leqslant \Delta}(R)$ (or equivalently $S_{> \Delta}(R)$), $W$ and the maps $(f_i)_{i \in W}$ for a suitable set $W$ depending on $R$, with $|W| \leqslant w_0(n)$. Now fix such an $n$; it follows that $|\mathcal{I}'(\mathcal{R}_n)|$ is at most the number of distinct triples $(S_{\leqslant \Delta}(R), W, (f_i)_{i \in W})$ arising from all racks in $\mathcal{R}_n$. There are clearly at most $2^n$ possibilities for each of $S_{\leqslant \Delta}(R)$ and $W$; as there are $n! \leqslant n^n$ choices for any map $f_i$, and $|W| \leqslant w_0(n)$, there are at most $n^{nw_0(n)}$ possibilities for the maps $(f_i)_{i \in W}$. Hence $|\mathcal{I}'(\mathcal{R}_n)| \leqslant 2^{2n}n^{nw_0(n)}$, and from \eqref{w0eq} we have that
\begin{align*}
\log_2 \left( 2^{2n} n^{nw_0(n)} \right) &= 2n + n (\log_2 n) w_0(n) \\
&= 2n + n (\log_2 n) \frac{7}{2} \frac{n}{(\log_2 n)^{3/2}} \\
& = 2n + \frac{7}{2}\frac{n^2}{(\log_2 n)^{1/2}} \\
& = o(n^2).
\end{align*} 
Hence for any $\epsilon > 0$, there exists a positive integer $n_1$ such that for $n \geqslant n_1$, $|\mathcal{I}'(\mathcal{R}_n)| \leqslant 2^{2n}n^{nw_0} \leqslant 2^{\epsilon n^2}$.
\end{proof}

\subsection{Components of the graph $G_{T(R)}$}

In order to prove that there are few choices for $\mathcal{I}''(R) = (\mathcal{I}_3(R), \ldots, \mathcal{I}_7(R))$, we will need the following lemma. Recall that $T(R) \subseteq S_{\leqslant \Delta}(R)$ is defined in Construction \ref{Tconst} and that $T^+(R) = T(R) \cup \Gamma_R^+(T(R))$.
\begin{lemma}
\label{limitways}
Let $R = ([n], \triangleright)$ be a rack. Let $C$ span a component of $G_{T(R)}$, and let $v \in C$. Then for any $j \in [n]$, knowledge of the maps $(f_l|_{T(R)})_{l \in [n]}$ and $(f_i)_{i \in T^+(R)}$ and the vertex $(v)f_j$ is sufficient to determine the map $f_j|_C$.
\end{lemma}
As in Definition \ref{calinfodef}, the maps $(f_l|_{T(R)})_{l \in [n]}$ determine the set $\Gamma^+_R(T(R))$ and thus the set $T^+(R)$.
\begin{proof}
Let $u \in C$; from Lemma \ref{dirpath} there is a directed $vu$-path in $G_{T(R)}$ (note that this graph is determined from the maps $(f_i)_{i \in T(R)}$, knowledge of which is assumed). Let $\overrightarrow{d}(v,u)$ denote the length of the shortest directed $vu$-path; we will show that $(u)f_j$ is determined by induction on the graph distance $d = \overrightarrow{d}(v,u)$. The base case $d=0$ is true by assumption, so take $d > 0$ and suppose the result is true for smaller $d$.

Take a shortest directed path (of length $d$) from $v$ to $u$ and let $w$ be the penultimate vertex on the path; then $\overrightarrow{wu}$ is an edge in $G_{T(R)}$, so there exists some $i \in T(R)$ such that $(w)f_i = u$. As we know $f_j|_{T(R)}$, $k := (i)f_j$ is determined; as $k \in T(R) \cup \Gamma_R^+(T(R)) = T^+(R)$, the map $f_k$ is also determined. Further, $\overrightarrow{d}(v,w) = d-1$, so by the inductive hypothesis the vertex $(w)f_j$ is determined. Now $f_k = f_{(i)f_j} = f_j^{-1}f_if_j$, and so $f_j = f_i^{-1}f_jf_k$; hence 
\[ (u)f_j = (u)f_i^{-1}f_jf_k = (w)f_jf_k, \]
which is determined as we know the vertex $(w)f_j$ and the map $f_k$. The result follows by induction. 
\end{proof}
We can now show the second main result of this section.
\begin{proposition}
\label{frakI2}
Let $\epsilon > 0$. Then there exists a positive integer $n_2$ such that for any $n \geqslant n_2$, $|\mathcal{I}''(\mathcal{R}_n)| \leqslant 2^{\epsilon n^2}$.
\end{proposition}
\begin{proof}
As in Corollary \ref{frakI1}, $|\mathcal{I}''(\mathcal{R}_n)|$ is equal to the number of distinct values of the 5-tuple $\mathcal{I}''(R)$ as $R$ ranges over all racks of order $n$; we will produce bounds for each of these entries. There are clearly at most $2^n$ choices for the set $T(R)$; recall that by construction $|T(R)| \leqslant L = \lfloor (\log_2 n)^2 \rfloor$ and that $T(R) \subseteq S_{\leqslant \Delta}(R)$, so $d_R^+(v) = |\Gamma_R^+(v)| \leqslant \Delta$ for any $v \in T(R)$. It follows that $|\Gamma_R^+(T(R))| \leqslant \sum_{v \in T(R)} d_R^+(v) \leqslant L\Delta$ and thus that $|T^+(R)| \leqslant L + L\Delta$. There are at most $n$ possibilities for the vertex $(u)f_i$ for any $i, u \in [n]$, so there are at most $(n^L)^n$ possibilities for the maps $(f_j|_{T(R)})_{j \in [n]}$ and at most $(n^n)^{L + L\Delta}$ possibilities for the maps $(f_i)_{i \in T^+(R)}$; hence there are at most $2^nn^{(2 + \Delta)Ln}$ possible values for the first three entries of $\mathcal{I}''(R)$ (as $R$ ranges over all racks of order $n$). 

Now consider a rack $R \in \mathcal{R}_n$ and suppose that the first three entries of $\mathcal{I}''(R)$ have been determined. Fix a $j \in S_{\leqslant \Delta}(R) \setminus T(R)$; we must consider the possibilities for the set $M_j$ of components of $G_{T(R)}$ merged by $\overrightarrow{E}_j$, and the restricted map $f_j|_C$ for each $C \in M_j$. If $|M_j| = a_j$ then there are (crudely) at most $n^{a_j}$ possibilities for $M_j$, as there are at most $n$ components of $G_{T(R)}$; from Lemma \ref{mergebound}, $a_j = |M_j| \leqslant 2n(\log_2 n)^{-2}$, so the number of possibilities for $M_j$ is at most
\[ \sum_{a_j = 0}^{ \lfloor 2n(\log_2 n)^{-2} \rfloor} n^{a_j} \leqslant \left( 1 + \frac{2n}{(\log_2 n)^2} \right) n^{2n(\log_2 n)^{-2}} \leqslant 3n \cdot n^{2n(\log_2 n)^{-2}}, \]
for sufficiently large $n$. Now take a $C \in M_j$ and choose an arbitrary $v \in C$; as the maps $(f_l|_{T(R)})_{l \in [n]}$ and $(f_i)_{i \in T^+(R)}$ have been determined already, we have from Lemma \ref{limitways} that the restriction $f_j|_C$ is determined entirely by $(v)f_j$. There are at most $n$ possibilities for this vertex, and so considering the $|M_j|$ components making up $Y_j$, there are at most $n^{|M_j|} \leqslant n^{2n(\log_2 n)^{-2}}$ possibilities for the restriction $f_j|_{Y_j}$.

Now note that $\mathbf{M}$ and $\triangleright|_Y$ are determined by $M_j$ and $f_j|_{Y_j}$ for each $j \in S_{\leqslant \Delta}(R) \setminus T(R)$; as there are at most $n$ such elements regardless of the set $S_{\leqslant \Delta}(R)$, there are at most $(3n)^n n^{2n^2(\log_2 n)^{-2}}$ possibilities for $\mathbf{M}$ and at most $n^{2n^2(\log_2 n)^{-2}}$ possibilities for $\triangleright|_Y$. Combining these bounds, there are at most 
\[ 2^nn^{(2 + \Delta)Ln}(3n)^n n^{4n^2(\log_2 n)^{-2}} = (6n)^n n^{(2 + \Delta)Ln + 4n^2(\log_2 n)^{-2}} =: \Lambda_n \]
possibilities for $\mathcal{I}''(R)$ as $R$ ranges over all racks of order $n$, for $n$ sufficiently large, and thus $|\mathcal{I}''(\mathcal{R}_n)| \leqslant \Lambda_n$. 

We have that $\Delta = (\log_2 n)^3$, so $2 + \Delta \leqslant 2(\log_2 n)^3$ for sufficiently large $n$. Thus for $n$ sufficiently large
\begin{align*}
\log_2 \left( \Lambda_n \right) &= n \log_2(6n) + (\log_2 n) \left( (2 + \Delta)Ln + \frac{4n^2}{(\log_2 n)^2} \right) \\
&\leqslant n \log_2 (6n) + 2n (\log_2 n)^6 + \frac{4n^2}{\log_2 n} \\
&= o(n^2).
\end{align*}
Hence for any $\epsilon > 0$, there exists a positive integer $n_2$ such that for $n \geqslant n_2$, $|\mathcal{I}''(\mathcal{R}_n)| \leqslant \Lambda_n \leqslant 2^{\epsilon n^2}$, proving the result.
\end{proof}

\section{Maps acting within components of $G_{T(R)}$}
\label{ressec}
\subsection{Some preparatory results}

We will need the following easy claim.

\begin{claim}
\label{calc}
For real numbers $x,y \geqslant 0$, $xy/3 + x^2/9 \leqslant (x+y)^2/8$.
\end{claim}
\begin{proof}
Simply observe that $(x + y)^2/8 - x^2/9 - xy/3 = (x - 3y)^2/72 \geqslant 0$.
\end{proof}
The above claim is used to prove the following key technical lemma. The notation is chosen to match the quantities defined in the next subsection.
\begin{lemma}
\label{n2vn}
Let $n$ be a positive integer and $(\eta_q)_{q=1}^n$ be a sequence of non-negative integers such that $\sum_{q=1}^n \eta_q = n$. Set
\[ \zeta = \left( \sum_{p=1}^n \frac{\eta_p}{p} \right) \left( \sum_{q=1}^n \frac{\log_2 q}{q} \eta_q \right). \]
Then $\zeta \leqslant n^2/4$.
\end{lemma}
\begin{proof}
By expanding the product, we have that
\begin{equation}
\label{deff}
\zeta = \sum_{q=1}^n \frac{\log_2 q}{q^2} \eta_q^2 + \sum_{p=1}^{n-1} \sum_{q = p+1}^n \frac{\log_2 p + \log_2 q}{pq} \eta_p\eta_q 
\end{equation}
and 
\[ \frac{n^2}{4} = \frac{(\eta_1 + \cdots + \eta_n)^2}{4} = \sum_{q=1}^n \frac{\eta_q^2}{4} + \sum_{p=1}^{n-1} \sum_{q = p+1}^n \frac{\eta_p\eta_q}{2}, \]
so if we set $\nu = n^2/4 - \zeta$, then
\[ \nu = \sum_{q=1}^n d_q \eta_q^2 + \sum_{p=1}^{n-1} \sum_{q = p+1}^n c_{p,q}\eta_p\eta_q, \] 
where $d_q := 1/4 - (\log_2 q)/q^2$ and $c_{p,q} := 1/2 - (\log_2 (pq))/pq$.

Since $2^r \geqslant r^2$ for all positive integers $r \neq 3$, $c_{p,q} \geqslant 0$ for $(p,q) \neq (1,3)$. Similarly, $2^{r^2} \geqslant r^4$ for all positive integers $r$ and thus $d_q \geqslant 0$ for all $q$; hence
\[ \nu \geqslant d_1\eta_1^2 + d_3\eta_3^2 + c_{1,3}\eta_1\eta_3. \]
We can bound this sum from below by using Claim \ref{calc} with $x = \eta_3$ and $y = \eta_1$; we have that
\begin{align*}
d_1\eta_1^2 + d_3\eta_3^2 + c_{1,3}\eta_1\eta_3 &= \frac{\eta_1^2}{4} + \left( \frac{1}{4} - \frac{\log_2 3}{9} \right)\eta_3^2 + \left( \frac{1}{2} - \frac{\log_2 3}{3} \right)\eta_1\eta_3 \\[3pt]
&= \frac{\eta_1^2}{4} + \frac{\eta_3^2}{4} + \frac{\eta_1\eta_3}{2} - (\log_2 3)\left( \frac{\eta_1\eta_3}{3} + \frac{\eta_3^2}{9} \right) \\[3pt]
&\geqslant \frac{(\eta_1 + \eta_3)^2}{4} - (\log_2 3) \frac{(\eta_1 + \eta_3)^2}{8} \\[3pt]
&\geqslant 0.
\end{align*}  
\end{proof}
A more elaborate version of this argument gives a corresponding stability result, saying (informally speaking) that $\zeta$ is close to $n^2/4$ if and only if $\eta_2$ is close to $n$ and $\eta_q$ is close to 0 for all $q \neq 2$; for full details see \cite{thesis}.

\subsection{Proof of Theorem \ref{mainthm}}
\label{tightss}

At the end of Section \ref{infosec}, we introduced the information $\mathcal{I}(R)$ in a rack $R$ on $[n]$, and explained how $\mathcal{I}$ can be thought of as a map from $\mathcal{R}_n$ to a set of 7-tuples; let us call this set $\mathfrak{I}_n$. In Section \ref{detsec} we showed that the image $\mathcal{I}(\mathcal{R}_n) \subseteq \mathfrak{I}_n$ has size at most $2^{o(n^2)}$; in this section, we will fix an $I$ in this image and consider all racks $R \in \mathcal{R}_n$ such that $\mathcal{I}(R) = I$. We will show that once the information corresponding to $I$ is known, there are not too many possibilities for $R$.

We will first consider the set $\mathfrak{I}_n$ of 7-tuples in more detail. From Definition \ref{calinfodef} and the subsequent discussion, each $I = (I_1, \ldots, I_7) \in \mathfrak{I}_n$ has the form described below.
\begin{enumerate}
\item $I_1$ is a set $S^I \subseteq [n]$;
\item $I_2$ is an $(n - |S^I|)$-tuple $(\sigma_i^I)$ of elements of $\Sym([n])$, indexed by $\tilde{S}^I := [n] \setminus S^I$ in some arbitrary order;
\item $I_3$ is a subset $T^I$ of $[n]$ such that $T^I \subseteq S^I$;
\item $I_4$ is an $n$-tuple $(\tau_i^I)_{i=1}^n$ of injective maps from $T^I$ to $S^I$;
\item $I_5$ is a sequence $(\sigma_i^I)$ of elements of $\Sym([n])$, indexed by the set $(T^I)^+ := T^I \cup (T^I)\tau_1^I \cup \cdots \cup (T^I)\tau_n^I$ in some arbitrary order.
\end{enumerate}
The last two entries of $\mathcal{I}(R)$ are graphical in nature; to relate them to an abstract $I \in \mathfrak{I}_n$ we will formally define the graph associated with such a 7-tuple $I$.
\begin{definition}
Let $I \in \mathfrak{I}_n$. Write $\Sigma^I = \lbrace \sigma_i^I \mid i \in T^I \rbrace$ and define $G_I := G_{\Sigma^I}$, in the sense of Definition \ref{multinvodef}, so $G_I$ is a $|T^I|$-edge-coloured multigraph on $[n]$. We write $c^I$ for $\cp(G_I)$ and $\mathcal{C}^I$ for the set of vertex sets of components of $G_I$, so $|\mathcal{C}^I| = c^I$. 
\end{definition} 
We can now describe the form of $I_6$ and $I_7$, namely:
\begin{enumerate}
\item[6.] $I_6$ is a sequence $(M_j^I)$ of subsets of $\mathcal{C}^I$, indexed by $S^I \setminus T^I$ in some arbitrary order;
\item[7.] $I_7$ is a sequence $(\psi_j^I)$ indexed by $S^I \setminus T^I$, where $Y_j^I = \bigcup_{C \in M_j^I} C$ and $\psi_j^I \in \Sym(Y_j^I)$ for all $j$.
\end{enumerate}
To avoid later inconvenience, we will extend the definition of $M_j^I$ and $Y_j^I$ to all $j \in [n]$ as follows. For each $j \in \tilde{S}^I \cup T^I$, define a set of ordered pairs from $[n]$ (which we can think of as edges of colour $j$) by setting
\begin{equation}
\label{EJdef}
\overrightarrow{E}_j^I := \lbrace \overrightarrow{xy} \mid (x)\sigma_j^I =y, \; x \neq y \rbrace.
\end{equation}
Then we define $M_j^I = M(G_I, \overrightarrow{E}_j^I)$ and $Y_j^I = \bigcup_{C \in M_j^I} C$.

Now fix an $I \in \mathfrak{I}_n$ and suppose $R \in \mathcal{R}_n$ is  a rack such that $\mathcal{I}(R) = I$. Recall that $\Delta = (\log_2 n)^3$ and
\[ \mathcal{I}(R) = \left(S_{\leqslant \Delta}(R), \; \triangleright|_{[n] \times S_{> \Delta}}, \; T(R), \; \triangleright|_{T(R) \times [n]}, \; \triangleright|_{[n] \times T^+(R)}, \; \mathbf{M}, \; \triangleright|_Y \right), \]
where each entry defined by $\triangleright$ can also be determined in terms of the maps $f_1, \ldots, f_n$ corresponding to $R$. Comparing $\mathcal{I}(R)$ with $I$ we see that $S_{\leqslant \Delta}(R) = S^I$, $S_{> \Delta}(R) = \tilde{S}^I$ and $T(R) = T^I$. We also have that $f_j = \sigma_j^I$ for $j \in \tilde{S}^I \cup (T^I)^+$, and thus $G_{T(R)} = G_I$. Finally, $M_j^I = M_j$ and $Y_j^I = Y_j$ for all $j \in [n]$, and $f_j|_{Y_j} = \psi_j^I$ for $j \in S^I \setminus T^I$.

This means that if the information $I = \mathcal{I}(R)$ is known, we need only determine the maps $(f_j|_{Z_j})_{j \in S_{\leqslant \Delta}(R) \setminus T(R)}$ to determine the entire rack $R$, noting that for each $j \in [n]$, the set $Z_j = [n] \setminus Y_j = [n] \setminus Y_j^I$ is determined by $I$. It follows that an upper bound on the number of possibilities for these maps is also an upper bound for the number of racks $R$ such that $\mathcal{I}(R) = I$.

We can further reduce the number of maps left to determine by considering components of the graph $G_I$.
\begin{lemma}
\label{1perC}
Let $I \in \mathfrak{I}_n$ be fixed and let $R \in \mathcal{R}_n$ be a rack such that $\mathcal{I}(R) = I$. Let $\mathcal{C}^I = \lbrace C_1, \ldots, C_{c^I} \rbrace$ and let $V = \lbrace v_1, \ldots, v_{c^I} \rbrace$ be a set of vertices such that $v_j \in C_j$ for each $j$. Then $R$ is determined by $I$ and the maps $(f_v|_{Z_v})_{v \in V}$.
\end{lemma}
\begin{proof}
As $\mathcal{I}(R) = I$, we have that $G_{T(R)} = G_I$ and so the set of (vertex sets of) components of $T(R)$ is $\mathcal{C}^I$. Take any $j = 1, \ldots, c^I$; the knowledge of $I$ determines the maps $(f_i)_{i \in T(R)} = (\sigma_i^I)_{i \in T^I}$ and thus from Lemma \ref{cruclem} (with $A = [n]$) knowledge of the map $f_{v_j}$ is sufficient to determine the maps $(f_u)_{u \in C_j}$. Applying this over all the components of $G_{T(R)}$ shows that we can determine the entire rack $R$ by determining the set of maps $(f_v)_{v \in V}$. As the restrictions $(f_v|_{Y_v})_{v \in V}$ are determined by $I$ ($f_v|_{Y_v} = f_v|_{Y_v^I}$ is equal to $\psi_v^I$ if $v \in S^I \setminus T^I$ and to $\sigma_v^I|_{Y_v^I}$ otherwise), we need only determine the restrictions $(f_v|_{Z_v})_{v \in V}$. 
\end{proof}
Before proving an upper bound we will need some more notation. For $I \in \mathfrak{I}_n$ and $1 \leqslant q \leqslant n$, let $\eta_q^I$ denote the number of vertices in components of $G_I$ of size exactly $q$. Then there are $\eta_q^I/q$ components of size exactly $q$ and thus
\[ n = \sum_{q=1}^n \eta_q^I \quad \text{and} \quad c^I = \sum_{q=1}^n \frac{\eta_q^I}{q}. \]
\begin{proposition}
\label{zetabound}
Let $I \in \mathfrak{I}_n$ and define
\[ \zeta^I = \left( \sum_{p=1}^n \frac{\eta_p^I}{p} \right) \left( \sum_{q=1}^n \frac{\log_2 q}{q} \eta_q^I \right). \]
Then there are at most $2^{\zeta^I}$ racks $R \in \mathcal{R}_n$ such that $\mathcal{I}(R) = I$.
\end{proposition}
\begin{proof}
Write $\mathcal{C}^I = \lbrace C_1, \ldots, C_{c^I} \rbrace$ and let $V = \lbrace v_1, \ldots, v_{c^I} \rbrace$ be a set of vertices such that $v_j \in C_j$ for each $j$. Let $R \in \mathcal{R}_n$ be a rack such that $\mathcal{I}(R) = I$; from Lemma \ref{1perC}, $R$ is determined by $I$ and the maps $(f_v|_{Z_v})_{v \in V}$. It follows that an upper bound on the number of possibilities for the maps $(f_v|_{Z_v})_{v \in V}$ is also an upper bound for the number of racks $R$ such that $\mathcal{I}(R) = I$.

Now fix a $v \in V$. Let $C \subseteq [n]$ span a component of $G_{T(R)} = G_I$ with $C \notin M_v$, and let $x \in C$; from Lemma \ref{YZsplit} (with $W = T(R)$), $(C)f_v = C$ and so in particular there are at most $|C|$ possibilities for $(x)f_v$. Now $(f_j|_{T(R)})_{j \in [n]} = (\tau_j^I)_{j \in [n]}$ and $(f_i)_{i \in T^+(R)} = (\sigma_i^I)_{i \in (T^I)^+}$ are determined by $I = \mathcal{I}(R)$, so from Lemma \ref{limitways} $f_v|_C$ is determined by $(x)f_v$. Thus there are at most $|C|$ possibilities for $f_v|_C$.

As there are $\eta_q^I/q$ components of $G_{T(R)} = G_I$ of size $q$, there are at most $N := \prod_{q=1}^n q^{\eta_q^I/q}$ possibilities for the map $f_v|_{Z_v}$. Considering all of the $c^I$ components together, there are at most $N^{c^I}$ possibilities for the maps $(f_v|_{Z_v})_{v \in V}$; it follows that there are at most $N^{c^I}$ racks $R$ such that $\mathcal{I}(R) = I$. Now  
\begin{align*}
\log_2 \left( N^{c^I} \right) &= c^I \sum_{q=1}^n \frac{\eta_q^I}{q} \log_2 q \\[3pt] 
&= \left(\sum_{p=1}^n \frac{\eta_p^I}{p} \right) \left( \sum_{q=1}^n \frac{\eta_q^I}{q} \log_2 q \right) \\[3pt]
&= \zeta^I,
\end{align*}
proving the result.
\end{proof} 
This proposition allows us to prove the main result.
\begin{proof}[Proof of Theorem \ref{mainthm}]
Recall that $\mathcal{R}'_n$ denotes the set of isomorphism classes of racks of order $n$; as $|\mathcal{R}'_n| \leqslant |\mathcal{R}_n|$ it suffices to find an upper bound on $|\mathcal{R}_n|$. Let $\epsilon > 0$; we have from Corollary \ref{frakI1} and Proposition \ref{frakI2} that for $n$ sufficiently large $|\mathcal{I}'(\mathcal{R}_n)| \leqslant 2^{\epsilon n^2/2}$ and $|\mathcal{I}''(\mathcal{R}_n)| \leqslant 2^{\epsilon n^2/2}$. Now take such a sufficiently large $n$; as $\mathcal{I}(R) = (\mathcal{I}'(R), \mathcal{I}''(R))$ for any $R \in \mathcal{R}_n$, $|\mathcal{I}(\mathcal{R}_n)| \leqslant |\mathcal{I}'(\mathcal{R}_n)||\mathcal{I}''(\mathcal{R}_n)| \leqslant 2^{\epsilon n^2}$, and so there are at most $2^{\epsilon n^2}$ possibilities for the 7-tuple $I \in \mathcal{I}(\mathcal{R}_n)$. From \mbox{Proposition \ref{zetabound}} and Lemma \ref{n2vn}, there are at most $2^{\zeta^I} \leqslant 2^{n^2/4}$ racks $R$ such that $\mathcal{I}(R) = I$, and thus $|\mathcal{R}_n| \leqslant 2^{\epsilon n^2} 2^{n^2/4}$.
\end{proof}

\section{An extremal result}

For each positive integer $n$, let $\mathcal{P}_n$ be a partition of $[n]$, and let $\mathcal{R}_{\mathcal{P}_n}$ denote the set of racks $R$ on $[n]$ such that the components of $G_R$ are exactly the parts of $\mathcal{P}_n$. Let $m_2(n)$ denote the number of parts of $\mathcal{P}_n$ of size exactly two; an extension of the methods used in this paper can be used to prove that unless $m_2(n) \sim n/2$ there exists a constant $0 < \kappa < 1/4$ such that $|\mathcal{R}_{\mathcal{P}_n}| \leqslant 2^{\kappa n^2}$ for infinitely many $n$. In other words, informally speaking, almost all (in an exponentially strong sense) racks $R$ on $[n]$ are such that almost all components of $G_R$ have size $2$. The idea of the proof is to find a function similar to $\zeta$ from Proposition~\ref{zetabound}, but taking into account the size of the components of $G_R$ rather than $G_{T(R)}$; the components of size two are a special case as the symmetric group on two elements is small and abelian. For a full proof, see \cite{thesis}.

\end{document}